\documentclass[preprint,12pt]{elsarticle}

\usepackage{amssymb}
\usepackage{amsmath}
\usepackage{hyperref}
\usepackage{dsfont}
\usepackage{graphicx}
\usepackage{tikz}
\usepackage{multirow}
\usepackage{subcaption}
\usepackage{float}
\usepackage{graphicx}
\usepackage{mathtools}
\usepackage{microtype}
\usepackage{multicol}
\usepackage{geometry}
\usepackage{tabularx}
\usepackage{tikz}
\usepackage{url}
\allowdisplaybreaks

\newtheorem{theorem}{Theorem}
\newtheorem{remark}{Remark}
\newtheorem{proof}{Proof}
\usepackage{booktabs} 

\journal{arXiv}

\begin{document}

\begin{frontmatter}



\title{A change-point problem for $m$-dependent multivariate random field}


\author[label1]{Vitalii Makogin\corref{cor2}} 
\author[label1]{Duc Nguyen\corref{cor1}}%
\ead{tran-1.nguyen@uni-ulm.de} 

\cortext[cor2]{Dedicated to the memory of Dr. Vitalii Makogin (12.12.1987 - 08.05.2024)}
\cortext[cor1]{Corresponding author}

\affiliation[label1]{organization={Institute of Stochastics, Ulm University},
            addressline={\\Helmholtzstraße 18}, 
            city={Ulm},
            postcode={89075}, 
            state={Baden-Württemberg},
            country={Germany}}

\begin{abstract}
In this paper, we consider a change-point problem for a centered, stationary and $m$-dependent multivariate random field. Under the distribution free assumption, a change-point test using CUSUM statistic is proposed to detect anomalies within a multidimensional random field, controlling the false positive rate as well as the Family-Wise Error in the multiple hypotheses testing context.
\end{abstract}



\begin{keyword}


CUSUM \sep multiple hypotheses testing \sep anomaly detection 
\end{keyword}

\end{frontmatter}



\section{Introduction}
The change-point problem is a fundamental issue with broad applications in various fields such as quality control, economics, medicine, and environmental science, to detect any abrupt variations. The change-point theory was developed thoroughly in \cite{page1954continuous, peach1995detection, Brodsky2013}. Its applications for different types of data such as time series, images are very numerous, see e.g. \cite{Alonso-Ruiz_Spodarev_2017, AlonsoRuiz_Spodarev_2018, Dresvyanskiy2020}, just to name a few. Among techniques for change-point detection, the Cumulative Sum (CUSUM) method, which is a sequential analysis device traditionally used in quality control for monitoring changes in the mean level of a process, has proven its simplicity, robustness, and effectiveness. In books \cite{brodsky1993nonparametric, Brodsky2013}, change-point problems were studied within a general parametric framework utilizing a CUSUM statistic test, which is widely recognized for anomaly detection, particularly in time series; see, for example, \cite{tartakovsky2014sequential, jones1970change}. 

In this paper, we use a CUSUM test to detect anomalies within a centered, stationary and $m$-dependent multivariate random field based on the methodology suggested in \cite{Dresvyanskiy2020}. The paper is organized as follows: In Section \ref{setting}, the model of anomaly detection for a random field is given, where the existence of anomaly is indicated through a hypothesis testing procedure. In Section \ref{main}, we obtain an upper bound for the tail probability of the test statistic, allowing one to get the global critical value by means of the type-I error is controlled. Section \ref{num} shows some numerical results based on the realizations of an $m$-dependent Gaussian random field.

\section{Problem Setting} \label{setting}

\subsection{Random field with change in mean}
For $n$ and $d$ positive integers, let $\left\{ \xi_k \in \mathbb{R}^n, k \in \mathbb{Z}^d \right\}$ be a centered, stationary, $m$-dependent and real-valued random field. For each $\theta \in \Theta$, define the corresponding set $I_\theta \in \mathbb{Z}^d$ as a set of scanning windows. Our goal is to detect anomalies in the space $\mathbb{Z}^d$ depending only on $\theta \in \Theta$.
We only consider anomalies in a window $W \subset \mathbb{Z}^d$, henceforth, $I_\theta \subset W$ and let $I^c_\theta = W \setminus I_\theta$. Assume that $W = [a_1, b_1] \times[a_2, b_2] \times \dots \times [a_d, b_d]$, where $a_i, b_i \in \mathbb{Z}^+$ for all $i = 1,2, \dots, d$. Nevertheless, it is important to have some restrictions on $I_\theta$, we need to consider some significant levels of the set $I_\theta$ to assure that whether it is a major part of the window $W$ or there is not an inconsiderable amount of anomalies. In other words, for $\gamma_1<\gamma_2 $ and $\gamma_1, \gamma_2 \in (0,1)$, from a practical point of view, we choose $\gamma_0 = 0.05$ and $\gamma_1 = 0.5$. Denote by
$$\Theta_0 = \left\{ \theta \in \Theta, \gamma_0 |W| \leq  |I_\theta| \leq \gamma_1 |W|\right\}.$$
Inversely, the set of parameters $\Theta_1 = \Theta \setminus \Theta_0$ represents for huge or extremely small of anomalies, i.e, 
$$\Theta_1 = \left\{ \theta \in \Theta,   |I_\theta| <\gamma_0 |W|, \ \text{or} \ |I_\theta| > (1-\gamma_1) |W|\right\}.$$

So far, we have defined the parametrized set of scanning windows imposed by some conditions. The presence of cracks can be studied by testing a null hypothesis, i.e., there is no cracks within the window $W$. Therefore, it is needed to construct a CUSUM test statistic and propose a rejection rule.
\subsection{Hypothesis testing}
Assume that there exists at least one anomaly region $I_{\theta_0}$ in the window $W$ and we observe values
$$s_k = \xi_k + h \mathds{1}\left\{k \in I_{\theta_0} \right\},$$
where $h$ is a fixed vector in $\mathbb{R}^n$ and unknown. The vector $h$ can be considered as the change in mean of $I_{\theta_0}$ and its complement $W \setminus I_{\theta_0}$. Hence, we consider the variation of the expectations of the random field $\left\{\xi_k, k\in\mathbb{Z}^d \right\}$ by testing the two following hypotheses:

$H_0: \mathbb{E}\xi_k = \mu$ for all $k \in W$, i.e., there no change in mean, versus

$H_1$: there exists a vector $h \in \mathbb{R}^n, h \neq \vec{0}$ in mean, i.e. $ \mathbb{E}\xi_k = \mu + h, k \in I_\theta$ and $ \mathbb{E}\xi_k = \mu , k \in I^c_\theta$\\
For a fixed window $W \subset \mathbb{Z}^d$, one has a sample $S = \left\{s_k \in \mathbb{R}^n, k\in\mathbb{Z}^d \right\}$ and
\begin{flalign}
	\begin{split}
		L(\theta) & = \frac{1}{|I_\theta|}\sum_{k \in I_\theta} s_k - \frac{1}{|I^c_\theta|}\sum_{k \in I^c_\theta} s_k\\
		& = \frac{1}{|I_\theta|}\sum_{k \in I_\theta}\left( \xi_k + h \mathds{1}\left\{k \in I_{\theta_0} \right\}\right) - \frac{1}{|I^c_\theta|}\sum_{k \in I^c_\theta} \left(\xi_k + h \mathds{1}\left\{k \in I_{\theta_0} \right\}\right)\\
		& = \frac{1}{|I_\theta|}\sum_{k \in I_\theta} \xi_k - \frac{1}{|I^c_\theta|}\sum_{k \in I^c_\theta} \xi_k +h \left(\frac{|I_\theta \cap I_{\theta_0}|}{|I_\theta|} - \frac{|I^c_\theta \cap I_{\theta_0}|}{|I^c_\theta|} \right).
	\end{split}
\end{flalign}
For testing purposes, employ the following statistic:
\begin{equation}\label{cusumstats}
	T_W(S) = \max_{\theta \in \Theta_0} \left\|L(\theta)\right\|_p =  \max_{\theta \in \Theta_0} \left\| \frac{1}{|I_\theta|}\sum_{k \in I_\theta} s_k - \frac{1}{|I^c_\theta|}\sum_{k \in I^c_\theta} s_k\right\|_p. 
\end{equation}
A rejection rule requires a threshold $y_\alpha$ such that if $T_W(S)$ exceeds $y_\alpha$, we reject the null hypothesis $H_0$. However, we need to compute the probability of type I error $\mathbb{P}_{H_0} \left(\max_{\theta \in \Theta_0} |L(\theta)| \geq y_\alpha \right)$. It yields,
\begin{align*}
	\begin{split}
		\mathbb{P}_{H_0} & \left(\max_{\theta \in \Theta_0} |L(\theta)| \geq y_\alpha \right)\\
		& = \mathbb{P}_{H_0} \left(\max_{\theta \in \Theta_0} \left| \frac{1}{|I_\theta|}\sum_{k \in I_\theta} \xi_k - \frac{1}{|I^c_\theta|}\sum_{k \in I^c_\theta} \xi_k +h \left(\frac{|I_\theta \bigcup I_{\theta_0}|}{|I_\theta|} - \frac{|I^c_\theta \bigcup I_{\theta_0}|}{|I^c_\theta|} \right)\right| \geq y_\alpha \right)\\
		& = \mathbb{P} \left(\max_{\theta \in \Theta_0} \left| \frac{1}{|I_\theta|}\sum_{k \in I_\theta} \xi_k - \frac{1}{|I^c_\theta|}\sum_{k \in I^c_\theta} \xi_k \right| \geq y_\alpha \right).
	\end{split}
\end{align*}
Denote by
$$\sum_{k \in W} b_k \xi_k  := \frac{1}{|I_\theta|}\sum_{k \in I_\theta} \xi_k - \frac{1}{|I^c_\theta|}\sum_{k \in I^c_\theta} \xi_k \quad \text{where} \quad b_k = \frac{\mathds{1}\left\{k \in I_\theta\right\}}{|I_\theta|} - \frac{\mathds{1}\left\{k \in I^c_\theta\right\}}{|I^c_\theta|}.$$ 
Note that if $\displaystyle b_k = \displaystyle\frac{\mathds{1}\left\{k \in I_\theta\right\}}{|I_\theta|} - \frac{\mathds{1}\left\{k \in I^c_\theta\right\}}{|I^c_\theta|}$, then
\begin{align*}
	\begin{split}
		\left\| b\right\|_1 = \sum_{k \in \mathbb{Z}^d} |b_k| & =   \sum_{k \in \mathbb{Z}^d}\left|\frac{\mathds{1}\left\{k \in I_\theta\right\}}{|I_\theta|} - \frac{\mathds{1}\left\{k \in I^c_\theta\right\}}{|I^c_\theta|}\right|\\
		& = \sum_{k \in I_\theta}\frac{\mathds{1}\left\{k \in I_\theta\right\}}{|I_\theta|} + \sum_{k \in I^c_\theta}\frac{\mathds{1}\left\{k \in I^c_\theta\right\}}{|I^c_\theta|} = 2\\
	\end{split}
\end{align*}
and
\begin{align*}\label{normb2}
	\begin{split}
		\left\| b\right\|_2^2 = \sum_{k \in \mathbb{Z}^d} |b_k|^2 & =   \sum_{k \in \mathbb{Z}^d}\left|\frac{\mathds{1}\left\{k \in I_\theta\right\}}{|I_\theta|} - \frac{\mathds{1}\left\{k \in I^c_\theta\right\}}{|I^c_\theta|}\right|^2\\
		& = \frac{1}{|I_\theta|} + \frac{1}{|I^c_\theta|} = \frac{|W|}{|I_\theta||I^c_\theta|}.\\
	\end{split}
\end{align*}
The main aim is to find a proper threshold $y_\alpha$ such that the probability of type I error of this statistic is less than $\alpha$, which can be $5\%$ or even $1\%$. Therefore, we need to find an upper bound $P$ for the above probability and then set it smaller than $5\%$, which means,
$$
\mathbb{P} \left(\max_{\theta \in \Theta_0} \left| \frac{1}{|I_\theta|}\sum_{k \in I_\theta} \xi_k - \frac{1}{|I^c_\theta|}\sum_{k \in I^c_\theta} \xi_k \right| \geq y_\alpha \right) \leq P = \alpha.
$$
\section{Main result}\label{main}

	In this section, we find an upper bound of the above tail probability. Let $\left\{\xi_k = (\xi_k^{(1)}, \xi_k^{(2)}, \dots, \xi_k^{(n)}) \in \mathbb{R}^n, k \in \mathbb{Z}^d \right\}$ be a random field. Then, for any $W \in \mathbb{Z}^d, |W| < \infty$, one has
	
\begin{align*}
	\mathbb{P}\left\{\left\| \sum_{k \in W}b_k \xi_k \right\|_p  \geq y \right\} & = \mathbb{P}\left\{\sum_{i = 1}^{n} \left|\sum_{k \in W} b_k \xi_k^{(i)} \right|^p \geq y^p \right\} \leq \sum_{i = 1}^{n} \mathbb{P}\left\{ \left|\sum_{k \in W} b_k \xi_k^{(i)} \right| \geq \frac{y}{n^{1/p}} \right\},
\end{align*}
for $1 \leq p < \infty$ and
\begin{equation}\label{pinfty}
	\mathbb{P}\left\{\left\| \sum_{k \in W}b_k \xi_k \right\|_\infty  \geq y \right\}  = \mathbb{P}\left\{\max_{i} \left|\sum_{k \in W} b_k \xi_k^{(i)} \right| \geq y \right\} \leq \sum_{i = 1}^{n} \mathbb{P}\left\{ \left|\sum_{k \in W} b_k \xi_k^{(i)} \right| \geq y \right\},
\end{equation}
for $p = \infty$.

Therefore, the upper bound for the tail inequality $\mathbb{P}\left\{\left\| \sum_{k \in W}b_k \xi_k \right\|_\infty  \geq y \right\}$ can be found by bounding from above the quantity $\mathbb{P}\left\{ \left|\sum_{k \in W} b_k \xi_k^{(i)} \right| \geq y \right\}$. Furthermore, since
$$
\mathbb{P}\left\{ \left|\sum_{k \in W} b_k \xi_k^{(i)} \right| \geq y \right\} \leq  \mathbb{P}\left\{ \left|\sum_{k \in W} b_k \xi_k^{(i)} \right| \geq \frac{y}{n^{1/p}} \right\},
$$
hence it is reasonable to bound from above the tail inequality for the case $p = \infty$.
This can be established using ideas from \cite{Dresvyanskiy2020,Heinrich1990}, we obtain a bound of multivariate $m$-dependent random field as follows:

\begin{theorem}\label{thr1}
	Let $\left\{\xi_k = (\xi_k^{(1)}, \xi_k^{(2)}, \dots, \xi_k^{(n)}) \in \mathbb{R}^n, k \in \mathbb{Z}^d \right\}$ be a mutivariate random field where $\{\xi_k^{(i)}  \in \mathbb{R}, k \in W \subset \mathbb{N}^d$ are stationary, m-dependent, real-valued univariate random fields. Assume there exists $H, \sigma >0$ such that
	\begin{equation} \label{assump}
	\mathbb{E}[\xi_k^{(i)}]^\ell \leq \frac{p!}{2}H^{p-2}\sigma, \ k \in W, \ i = 1,2,\dots n,
	\end{equation}
	then
	\begin{align*} 
		\mathbb{P}_{H_0}\left\{\left\| \sum_{k \in W}b_k \xi_k \right\|_\infty  \geq y \right\} & \leq 2n  \exp\left\{\frac{-y^2 }{4m^d\sigma^2} \frac{|I_\theta||I_\theta^c|}{|W|}\right\} \mathds{1}\left\{|I_\theta^c| \leq \frac{\sigma^2|W|}{y H}\right\}\\
		&  +2n \exp\left\{\frac{-y |I_\theta|}{2Hm^d} + \frac{\sigma^2 |W||I_\theta|}{4H^2m^d |I_\theta^c|}\right\} \mathds{1}\left\{|I_\theta^c| > \frac{\sigma^2|W|}{y H}\right\}.
	\end{align*}
\end{theorem}

\begin{proof} Under the same assumptions, the upper bound of the tail probability for univariate case are given in \cite{Dresvyanskiy2020}, where
\begin{align*} 
		\mathbb{P}_{H_0}\left\{ \left|\sum_{k \in W} b_k \xi_k^{(i)} \right| \geq y \right\} & \leq 2  \exp\left\{\frac{-y^2 }{4m^d\sigma^2} \frac{|I_\theta||I_\theta^c|}{|W|}\right\} \mathds{1}\left\{|I_\theta^c| \leq \frac{\sigma^2|W|}{y H}\right\}\\
		&  +2 \exp\left\{\frac{-y |I_\theta|}{2Hm^d} + \frac{\sigma^2 |W||I_\theta|}{4H^2m^d |I_\theta^c|}\right\} \mathds{1}\left\{|I_\theta^c| > \frac{\sigma^2|W|}{y H}\right\}.
	\end{align*}
Therefore, from (\ref{pinfty}), 
\begin{align*} 
		\mathbb{P}_{H_0}\left\{ \left\| \sum_{k \in W}b_k \xi_k \right\|_\infty  \geq y \right\} & \leq \sum_{i = 1}^{n} \mathbb{P}_{H_0}\left\{ \left|\sum_{k \in W} b_k \xi_k^{(i)} \right| \geq y \right\}\\
		& \leq 2n  \exp\left\{\frac{-y^2 }{4m^d\sigma^2} \frac{|I_\theta||I_\theta^c|}{|W|}\right\} \mathds{1}\left\{|I_\theta^c| \leq \frac{\sigma^2|W|}{y H}\right\}\\
		& \qquad +2n \exp\left\{\frac{-y |I_\theta|}{2Hm^d} + \frac{\sigma^2 |W||I_\theta|}{4H^2m^d |I_\theta^c|}\right\} \mathds{1}\left\{|I_\theta^c| > \frac{\sigma^2|W|}{y H}\right\}.
	\end{align*}

\end{proof}

\begin{remark}
In the case of univariate random field, i.e., $n = 1$, the upper bound coincides with the result from \cite{Dresvyanskiy2020}.
\end{remark}

\begin{theorem} 
	Let $\left\{\xi_k = (\xi_k^{(1)}, \xi_k^{(2)}, \dots, \xi_k^{(n)}) \in \mathbb{R}^n, k \in \mathbb{Z}^d \right\}$ be a mutivariate random field where $\{\xi_k^{(i)}  \in \mathbb{R}, k \in W \subset \mathbb{N}^d\}, \ i=1,\ldots, n$ are stationary, m-dependent, real-valued univariate random field. Assume there exists $H, \sigma >0$ such that the inequality (\ref{assump}) holds, then
		\begin{align} \label{eq1}
		\mathbb{P}_{H_0} \left(\max_{\theta \in \Theta_0} |L(\theta)| \geq y_\alpha \right) & \leq 2n \displaystyle\sum_{\theta \in \Theta_0,  |I_\theta| \leq \frac{\sigma^2|W|}{yH}} \exp\left\{\frac{-y^2 }{4m^d\sigma^2} \frac{|I_\theta||I_\theta^c|}{|W|}\right\} \nonumber\\
		&  +2n \displaystyle\sum_{\theta \in \Theta_0,  |I_\theta| > \frac{\sigma^2|W|}{yH}} \exp\left\{\frac{-y |I_\theta|}{2Hm^d} + \frac{\sigma^2 |W||I_\theta|}{4H^2m^d |I_\theta^c|}\right\} .
	\end{align}
\end{theorem}

From (\ref{eq1}), the null hypothesis will be rejected with type-I error controlled at level $\alpha$ if the critical value $y$ is obtained by solving the following equation:
\begin{align}\label{eq:alpha}
	2n \displaystyle\sum_{\theta \in \Theta_0,  |I_\theta| \leq  \frac{\sigma^2|W|}{yH}}& \exp\left\{\frac{-y^2 }{4m^d\sigma^2} \frac{|I_\theta||I_\theta^c|}{|W|}\right\} \nonumber\\
	&  +2n \displaystyle\sum_{\theta \in \Theta_0,  |I_\theta| > \frac{\sigma^2|W|}{yH}} \exp\left\{\frac{-y |I_\theta|}{2Hm^d} + \frac{\sigma^2 |W||I_\theta|}{4H^2m^d |I_\theta^c|}\right\} = \alpha.
\end{align}
\begin{proof}
Using the Bonferroni correction \cite{bonferroni1936teoria}, one has
\begin{align*}
\mathbb{P}_{H_0} \left(\max_{\theta \in \Theta_0} |L(\theta)| \geq y_\alpha \right) & \leq \sum_{\theta \in \Theta_0} \mathbb{P}_{H_0} \left( |L(\theta)| \geq y_\alpha \right)\\
& \leq 2n  \sum_{\theta \in \Theta_0}\exp\left\{\frac{-y^2 }{4m^d\sigma^2} \frac{|I_\theta||I_\theta^c|}{|W|}\right\} \mathds{1}\left\{|I_\theta^c| \leq \frac{\sigma^2|W|}{y H}\right\} \\
& \quad + 2n \sum_{\theta \in \Theta_0} \exp\left\{\frac{-y |I_\theta|}{2Hm^d} + \frac{\sigma^2 |W||I_\theta|}{4H^2m^d |I_\theta^c|}\right\} \mathds{1}\left\{|I_\theta^c| > \frac{\sigma^2|W|}{y H}\right\},
\end{align*}
where the last inequality is obtained using Theorem \ref{thr1}.
\end{proof}

\begin{remark} If $\left\{\xi_k = (\xi_k^{(1)}, \xi_k^{(2)}, \dots, \xi_k^{(n)}) \in \mathbb{R}^n, k \in \mathbb{Z}^d \right\}$ is a multivariate Gaussian random field where $\mathbb{E}[\xi^{(i)}_0]^p \leq \sigma^2, i = 1, \dots, n$, the choice $H = \sigma$ is suitable due to the fact that $\mathbb{E}|\xi^{(i)}_0|^p \leq \sigma^{p-2} \sigma^2 \mathbb{E}|Z|^p, p \geq 2$ where $Z \sim \mathcal{N}(0,1)$. 
\end{remark}

\begin{remark} By comparing the maximum of all statistics with a certain threshold, one can test the global hypothesis $H_0$, i.e., there exist anomalies within the whole random field. Furthermore, if for any $\theta \in \Theta$, we want to test the null hypothesis "$H_0(\theta)$: $I_\theta$ is homogeneous" versus "$H_1(\theta):$ $I_\theta$ contains anomalies", the global theoretical critical value $y_\alpha$ defines a testing procedure in which controls the Family-Wise Error ($FWER$), where
$$
FWER = \mathbb{P}(\text{reject falsely at least one true hypothesis}) = \mathbb{P}_{H_0} \left(\max_{\theta \in \Theta_0} |L(\theta)| \geq y_\alpha \right).
$$
\end{remark}

\section{Numerical results} \label{num}

In this section, we study the empirical distribution of the statistic $T_W(S)$ given in (\ref{cusumstats}) and the behavior of the tail probability provided in (\ref{eq1}) under different settings of $\sigma$ and $m$. 

Let $W = [1, P] \times [1, Q] \times [1, R] \bigcap \mathbb{N}^3$. The parametrized collection of scanning windows is defined as follows:
\begin{align*}
\Theta_0 := &\{ \theta = (a_1, a_2, a_3, s_1, s_2, s_3) \in \mathbb{N}^6;  \\
& \qquad 1+ a_1+s_1 \leq P, 1+a_2+s_2 \leq Q, 1 + a_3+s_3\leq R;  \gamma_0 \leq \frac{s_1 s_2 s_3}{W} \leq \gamma_1\}.
\end{align*}

For each $\theta \in \Theta$, the scanning window $I_\theta = [1+a_1,1+a_1+s_1] \times [1+a_2,1+a_2+s_2] \times [1+a_3,1+a_3+s_3]$ is set within $W$. We choose $\gamma_0 = 0.05, \gamma_1 = 0.5$ to avoid the case that the scanning windows are either excessively small or excessively large relative to anomalies, where the scanning windows flatten out the statistical significance. 

We generate 500 realizations of a centered, $m$-dependent, multivariate Gaussian random field $\{\xi_k = (\xi_k^{(1)}, \xi_k^{(2)}, \xi_k^{(3)}) \in \mathbb{R}^3, k \in W \}$ , where $W = [1,50] \times [1,50] \times [1,50], \xi_k^{(i)} \sim \mathcal{N}(0,1), k \in W, i = 1,2,3$. With $m = 5, 7$, the $m$-dependence assumption can be achieved by setting $Y_{1+m\ell}, \ell \in W$ are independent and $Y_{1+m\ell} = Y_{r + m\ell}, r \in \{1, \ldots, m\}^3$. The set of scanning windows $\{I_\theta, \theta \in \Theta_0\}$ is defined where all $I_\theta$ are equal cubic of the size $30 \times 30$, satisfying the requirements of the parametric setting. In this case, $|\Theta_0| = 9261$.

\begin{table}[ht]
  \centering
  \small
  \begin{tabular}{lcccccccc}
    \toprule
    & \(\sigma^2 = 0.5\) & \(\sigma^2 = 0.6\) & \(\sigma^2 = 0.7\) & \(\sigma^2 = 0.8\) & \(\sigma^2 = 0.9\) & \(\sigma^2 = 1.0\) & \(\sigma^2 = 1.1\)\\
    \midrule
    \(m = 3\) & 0.6009 & 0.6424 & 0.6814 & 0.7183 & 0.7533 & 0.7868 & 0.8191  \\
    \(m = 4\) & 0.8398 & 0.8978 & 0.9522 & 1.0038 & 1.0528 & 1.1001 & 1.1459  \\
    \(m = 5\) & 1.1861 & 1.2675 & 1.3490 & 1.4305 & 1.5120 & 1.5935 & 1.6750 \\
    \(m = 6\) & 1.5693 & 1.6987 & 1.8281 & 1.9575 & 2.0869 & 2.2163 & 2.3457  \\
    \(m = 7\) & 2.0794 & 2.2725 & 2.4656 & 2.6588 & 2.8519 & 3.0450 & 3.2381  \\
    \(m = 8\) & 2.7342 & 3.0092 & 3.2842 & 3.5593 & 3.8343 & 4.1093 & 4.3843  \\
    \(m = 9\) & 3.5521 & 3.9293 & 4.3066 & 4.6838 & 5.0610 & 5.4382 & 5.8155  \\
    \(m = 10\) & 4.5230 & 5.0217 & 5.5205 & 6.0193 & 6.5180 & 7.0168 & 7.5156  \\
    \bottomrule
  \end{tabular}
    \caption{Critical value $y$ from (\ref{eq:alpha}) with $\alpha = 0.05$}
        \label{yalpha}
\end{table}

Let $\alpha = 0.05$, we compare the empirical critical value $\hat{y}_{0.05}$ with $y_{0.05}$ computed from Equation (\ref{eq1}). For $m = 5$ and $m = 7$, one has $\hat{y}_{0.05} = 0.5369$ and $\hat{y}_{0.05} = 0.7259$, respectively. The critical values $y_{0.05}$ with different values of $m$ and $\sigma$ are shown in Table \ref{yalpha}. For instance, since the exact values of $m$ are known, one has the theoretical critical values $y_{0.05} = 1.5935$ and  $y_{0.05} = 2.8519$, leading the test to be conservative. However, from Table \ref{yalpha}, the theoretical critical value is quite sensitive to any changes of $m$, suggesting to choose a smaller $m$ in the computation of $y_{0.05}$, such as $m = 3$, resulting in $y_{0.05} = 0.7868$, therefore controlling the type-I error at the desired level $\alpha =0.05$.

When $m$ is unknown, it can be estimated by statistically assessing the covariance function of the random field. The value of $m$ should be selected such that the empirical covariance function is sufficiently close to zero. 

\section{Conclusion}

In this paper, we generalized the results given in \cite{Dresvyanskiy2020} to the case of multivariate random fields. The results show that by properly choosing the value of $m$, one is able to test the global hypothesis $H_0$, i.e., if the observed random field contains anomalies, at the pre-determined significance level $\alpha$. The main challenge of this approach is that one has to employ the Bonferroni correction \cite{bonferroni1936teoria} due to the dependence between observations from different scanning windows, which usually makes the test extremely conservative, especially when $|\Theta_0|$ is large.

Nevertheless, in the context of multiple hypotheses testing, the obtained upper bound for the tail probability provides a global threshold for each statistic, controlling the Family-Wise Error rate at the significance level $\alpha$, allowing one to localize anomalies within a random field. 

\section*{Acknowledgement} We would like to extend our sincere gratitude to  Professor. Dr. Evgeny Spodarev, for his valuable guidance, insightful feedback, and continuous support throughout the course of this research.

\section*{Funding}
This research was funded by the German Federal Ministry of Education and Research (BMBF) [grant number 05M20VUA (DAnoBi)].

\section*{Declaration of Competing Interest}
The authors declare that they have no known competing financial interests or personal relationships that could have appeared to influence the work reported in this paper.

\end{document}